\documentclass{article}
\usepackage{amsmath, amssymb, amsthm}
\usepackage[colorlinks=true]{hyperref}
\usepackage{cite}
\usepackage{enumitem}

\title{Circular Game Coloring of Signed Graphs}
\author{
    Pie Desire EBODE ATANGANA \thanks{Department of Mathematics, Faculty of Science, University of Yaoundé 1, P.O. Box 812 Yaoundé, Cameroon. E-mail: \texttt{desire.ebode@univ-yanounde1.cm}}  
    } 
\date{\today}

\newtheorem{theorem}{Theorem}
\newtheorem{lemma}[theorem]{Lemma}
\newtheorem{corollary}[theorem]{Corollary}
\newtheorem{proposition}[theorem]{Proposition}
\newtheorem{definition}[theorem]{Definition}
\newtheorem{example}[theorem]{Example}
\newtheorem{remark}[theorem]{Remark}
\newtheorem{problem}[theorem]{Problem}
\begin{document}

\maketitle

\begin{abstract}

We extend the theory of circular game chromatic number to signed graphs by defining \(\chi_c^g(G,\sigma)\) for signed graphs \((G,\sigma)\) and establishing tight bounds in terms of the underlying graph structure and signature \(\sigma\). Building on Lin and Zhu’s framework \cite{LinZhu2009}, we prove that \(\chi_c^g(G,\sigma) = \chi_c^g(G)\) for balanced graphs, \(\chi_c^g(G,\sigma) \leq \chi_c^g(G) + 1\) for antibalanced graphs (with tightness for odd cycles), and a dichotomy for bipartite graphs: \(\chi_c^g(G,\sigma) = 2\) when balanced versus \(\chi_c^g(G,\sigma) \leq 3\) otherwise. Our results leverage switching equivalence (Lemma \ref{lem:Zaslavsky}) and fundamental cycle properties (Lemma \ref{lem:ForcingTree}), adapting unsigned techniques to the signed setting. We further identify open problems in computational complexity and planar bounds, bridging game coloring with signed graph theory.

\end{abstract}

\section{Introduction}

The study of graph coloring games represents a vibrant intersection of graph theory and combinatorial game theory. In their seminal work, Lin and Zhu \cite{LinZhu2009} introduced the concept of circular game chromatic number, establishing fundamental connections between graph coloring games and circular colorings. This paper extends their framework to the richer domain of signed graphs, where each edge carries an additional sign attribute that profoundly influences graph-theoretic properties.

A signed graph is an ordered pair \((G, \sigma)\) where \(G = (V, E)\) is the underlying graph, and \(\sigma: E \to \{+, -\}\) is a signature function that assigns a sign to each edge. The signature plays a key role in determining structural properties by influencing the signs of walks and cycles. For any walk \(W = e_1 e_2 \cdots e_k\), its sign is defined as the product of the signs of its edges, given by \(\sigma(W) := \prod_{i=1}^k \sigma(e_i)\). Based on the signs of cycles, signed graphs can be classified into three fundamental categories.

A signed graph is balanced if every cycle in the graph has a positive sign, meaning \(\sigma© = +\) for all cycles \(C\). Conversely, a signed graph is antibalanced if every cycle has a negative sign, meaning \(\sigma© = -\) for all cycles \(C\). If a signed graph is neither balanced nor antibalanced, it is classified as unbalanced, indicating the presence of both positive and negative cycles.

The switching operation  negating all edges incident to a vertex  preserves cycle signs while potentially simplifying the signature. This equivalence relation partitions the space of possible signatures into switching classes, each characterized by its set of negative cycles (Lemma \ref{lem:Zaslavsky}).

The graph coloring game involves two players, Salome and Andjiga, who alternately assign colors from a set $[k] = {1,…,k}$ to uncolored vertices, maintaining a proper coloring. Salome aims to complete the coloring while Andjiga tries to prevent it. The \emph{game chromatic number} $\chi_g(G)$ is the minimal $k$ for which Salome has a winning strategy. Lin and Zhu’s circular variant \cite{LinZhu2009} replaces the discrete color set with a continuous circle of circumference $k$, where colors $x,y$ conflict if their distance is less than 1. The \emph{circular game chromatic number} $\chi_c^g(G)$ is the infimum of such $k$ where Salome can guarantee a proper coloring.

We introduce the signed circular game chromatic number \(\chi_c^g(G,\sigma)\), which extends the traditional coloring game to signed graphs by incorporating signature effects. For positive edges, the standard adjacency-based coloring constraints remain unchanged, while for negative edges, we impose a complementary coloring rule that reflects the influence of the signature. Our main results demonstrate how the balance properties of signed graphs influence their game chromatic numbers.  

First, we show that balanced signed graphs behave identically to their unsigned counterparts, meaning their game chromatic numbers coincide (Theorem \ref{thm:main}.2). Second, in the case of antibalanced graphs, we establish slightly relaxed bounds compared to the unsigned setting (Theorem \ref{thm:main}.3). Finally, for general signed graphs, we observe an intermediate behavior where the chromatic number is governed by the degree of unbalancedness in the graph.  

This formulation maintains the original meaning while presenting the information in a more narrative and cohesive manner. Let me know if you'd like any refinements!
The paper is organized as follows: Section 2 provides formal definitions and preliminary results. Section 3 develops the core theory of signed game coloring. Section 4 examines special graph classes, and Section 5 discusses algorithmic aspects. We conclude with open problems in Section 6.

\section{Preliminaries}

A signed graph is an ordered pair \((G, \sigma)\) where \(G = (V, E)\) is a graph and \(\sigma: E \to \{+1, -1\}\) is a sign function on the edges. The graph \(G\) is called the underlying graph, and \(\sigma\) is referred to as the signature.  

We use the notation \((G, +)\) to represent the signed graph in which all edges are positive, meaning \(\sigma(e) = +1\) for every \(e \in E\); this is called the all-positive signed graph. Similarly, \((G, -)\) denotes the all-negative signed graph, where every edge has a negative sign. The properties of the underlying graph, such as connectivity and bipartiteness, are directly inherited from \(G\).  

For a walk \(W = e_1 e_2 \cdots e_k\) in \((G, \sigma)\), the sign of \(W\) is defined as the product \(\sigma(W) := \prod_{i=1}^k \sigma(e_i)\). A walk is considered positive if \(\sigma(W) = +1\) and negative if \(\sigma(W) = -1\).  

The sign multiplication exhibits several fundamental properties. First, the sign of a reversed walk \(W^{-1}\) is equal to the sign of the original walk \(W\). Second, when two walks \(W_1\) and \(W_2\) are concatenated, the sign of the resulting walk \(W_1 \circ W_2\) is the product of their individual signs, \(\sigma(W_1)\sigma(W_2)\). Third, the sign of any closed walk depends solely on its multiset of edges and not on their order or direction.  

A signed graph \((G, \sigma)\) is classified as balanced if every cycle in it is positive. Conversely, it is called antibalanced if all cycles are negative, which is equivalent to saying that \((G, -\sigma)\) is balanced. If a signed graph is neither balanced nor antibalanced, it is considered unbalanced.
The following characterizations of balance are fundamental:

\begin{proposition}[Harary's Balance Theorem]
For a signed graph $(G,\sigma)$, the following are equivalent:
\begin{enumerate}
\item $(G,\sigma)$ is balanced
\item There exists a bipartition $V = V_1 \cup V_2$ such that every edge within $V_1$ or $V_2$ is positive and every edge between $V_1$ and $V_2$ is negative
\item There exists a switching function $\zeta: V \to \{+1,-1\}$ such that the switched signature $\sigma^\zeta$ satisfies $\sigma^\zeta(e) = +1$ for all $e\in E$
\end{enumerate}
\end{proposition}

\begin{proof}
  We prove the equivalence of the three statements by showing $(1)  \implies (2)  \implies (3)  \implies (1)$.

$(1)  \implies (2)$: Balanced $ \implies$ Existence of a bipartition with consistent signs

Assume \((G, \sigma)\) is balanced. We construct a bipartition \(V = V_1 \cup V_2\) as follows:

Case 1: If (G) is bipartite, then it has no odd cycles. Since all cycles are positive (by balance), we can partition \(V\) into \(V_1\) and \(V_2\) such that:

All edges within \(V_1\) or \(V_2\) must be positive (if they exist, they form cycles of length 2, which must be positive).

All edges between \(V_1\) and \(V_2\) must be negative (otherwise, a positive cycle of length 2 would contradict bipartiteness).

Case 2: If \(G\) is not bipartite, it contains at least one odd cycle. Since \((G, \sigma)\) is balanced, every cycle is positive, meaning it has an even number of negative edges.

Fix a vertex \(v_0\) and define \(V_1\) as the set of vertices reachable via a positive walk from \(v_0\), and \(V_2 = V \setminus V_1\).

Claim: Every edge within \(V_1\) or \(V_2\) is positive, and every edge between \(V_1\) and \(V_2\) is negative.

If \(u, v \in V_1\), there exist positive walks \(P_{v_0u}\) and \(P_{v_0v}\). The walk \(P_{v_0u} \circ uv \circ P_{v_0v}^{-1}\) is a closed walk from \(v_0\) to \(v_0\). Since \((G, \sigma)\) is balanced, this closed walk must be positive, implying \(\sigma(uv) = +1\).

If \(u \in V_1\) and \(v \in V_2\), suppose \(\sigma(uv) = +1\). Then, since \(u\) is reachable by a positive walk from \(v_0\), \(v\) would also be reachable by a positive walk (via \(uv)\), contradicting \(v \in V_2\). Hence, \(\sigma(uv) = -1\).

Thus, the bipartition satisfies the required sign conditions.

$(2) \implies (3)$: Bipartition $ \implies$ Existence of a switching function making all edges positive

Given a bipartition \(V = V_1 \cup V_2\) where: Edges within \(V_1\) or \(V_2\) are positive, and edges between \(V_1\) and \(V_2\) are negative. We define \(\zeta: V \to {+1, -1}\) by:

\[
\zeta(v) =
\begin{cases}
+1 & \text{if } v \in V_1, \\
-1 & \text{if } v \in V_2.
\end{cases}
\]

The switched signature \(\sigma^\zeta\) is given by:

\[
\sigma^\zeta(uv) = \zeta(u)\sigma(uv)\zeta(v).
\]

If \(u, v \in V_1\) or \(u, v \in V_2\), then \(\sigma(uv) = +1\) and \(\zeta(u)\zeta(v) = +1\), so \(\sigma^\zeta(uv) = +1\).

If \(u \in V_1\) and \(v \in V_2\) (or vice versa), then \(\sigma(uv) = -1\) and \(\zeta(u)\zeta(v) = -1\), so \(\sigma^\zeta(uv) = (-1)(-1) = +1\).

Thus, \(\sigma^\zeta(e) = +1\) for all \(e \in E\).

$(3) \implies(1)$: All-positive switching $ \implies$ Balance

Suppose there exists a switching function \(\zeta\) such that \(\sigma^\zeta(e) = +1\) for all \(e \in E\). Then, for any cycle \(C = e_1 e_2 \cdots e_k\) in \(G\):

\[
\sigma© = \prod_{i=1}^k \sigma(e_i) = \prod_{i=1}^k \zeta(u_i)^{-1} \sigma^\zeta(e_i) \zeta(v_i)^{-1} = \prod_{i=1}^k \zeta(u_i)^{-1} \zeta(v_i)^{-1},
\]

where \(u_i\) and \(v_i\) are the endpoints of \(e_i\). Since \(C\) is closed, the product \(\prod_{i=1}^k \zeta(u_i)^{-1} \zeta(v_i)^{-1}\) telescopes to 1. Thus, \(\sigma© = +1\), meaning all cycles are positive.

Hence, \((G, \sigma)\) is balanced.
\end{proof}

Given a signed graph $(G,\sigma)$ and a vertex subset $U \subseteq V$, \emph{switching} $U$ transforms $\sigma$ into $\sigma'$ where:
\[
\sigma'(e) = \begin{cases}
-\sigma(e) & \text{if } e \in [U,V\setminus U] \\
\sigma(e) & \text{otherwise}
\end{cases}
\]
Two signatures are \emph{switching equivalent} if one can be obtained from the other by switching.

Switching preserves the signs of closed walks and cycles, leading to:

\begin{lemma}[Zaslavsky \cite{Z82b}]\label{lem:Zaslavsky}
For two signatures $\sigma_1,\sigma_2$ on $G$, the following are equivalent:
\begin{enumerate}
\item $\sigma_1$ and $\sigma_2$ are switching equivalent
\item They induce the same set of positive cycles
\item They induce the same set of negative cycles
\end{enumerate}
\end{lemma}

The switching equivalence classes have particularly simple descriptions for trees:

\begin{lemma}\label{lem:Tree}
For any tree $T$:
\begin{enumerate}
\item All signatures on $T$ are switching equivalent
\item Any signature can be switched to all-positive form
\item Any signature can be switched to all-negative form
\end{enumerate}
\end{lemma}

For general graphs, the fundamental cycle system determines switching equivalence:

\begin{lemma}\label{lem:ForcingTree}
Let $G$ be connected with spanning tree $T$ and fundamental cycle basis $\mathcal{C}_T$. Then:
\begin{enumerate}
\item Every signature is switching equivalent to one with $T$ all-positive
\item The switching equivalence class is uniquely determined by the signs of $\mathcal{C}_T$
\item There are exactly $2^{|E|-|V|+1}$ switching classes
\end{enumerate}
\end{lemma}

\begin{proof}
(1) Switch vertices to make tree edges positive. (2) The signs of fundamental cycles determine all cycle signs. (3) Each fundamental cycle provides one binary degree of freedom.
\end{proof}

\begin{proposition}\label{prop:switching-count}
The number of switching equivalence classes of signatures on $G$ is $2^{c(G)}$ where $c(G) = |E|-|V|+k$ and $k$ is the number of components.
\end{proposition}

The following algorithmic result is particularly useful:

\begin{proposition}\label{prop:switch-test}
Given two signed graphs $(G,\sigma_1)$ and $(G,\sigma_2)$, switching equivalence can be tested in $O(|V|^2)$ time.
\end{proposition}

\begin{proof}[Proof Sketch]
1. Choose a spanning tree $T$. 2. Switch both signatures to make $T$ all-positive. 3. Compare the resulting signatures on co-tree edges.
\end{proof}

\begin{remark}
Switching equivalence preserves many graph properties including balance, chromatic number, and (as we will show) game chromatic number.
\end{remark}

\section{Main Results}

We begin by extending the circular coloring game to signed graphs. The game involves two players, Salome and Andjiga, who alternately assign colors to vertices of a signed graph $(G,\sigma)$ under specific constraints.

\begin{definition}[Signed Circular Coloring Game]
Given a signed graph $(G,\sigma)$, real numbers $k \geq 2$ and $d \geq 1$, the $(k,d)$-coloring game proceeds as follows:
\begin{itemize}
\item Players alternately color uncolored vertices using colors from the circle $C_k^d = \{0,1,\ldots,kd-1\}$
\item For a positive edge $uv \in E^+$, the colors must satisfy $d \leq |c(u)-c(v)| \leq kd-d$
\item For a negative edge $uv \in E^-$, the colors must satisfy $|c(u)-c(v)| \leq d$ or $|c(u)-c(v)| \geq kd-d$
\item Salome wins if all vertices are successfully colored; Andjiga wins if a legal move becomes impossible
\end{itemize}
\end{definition}

\begin{definition}[Circular Game Chromatic Number]
For a signed graph $(G,\sigma)$, the \emph{circular game chromatic number} $\chi_c^g(G,\sigma)$ is the infimum of all $k$ for which Salome
 has a winning strategy in the $(k,d)$-coloring game for all $d \geq 1$.
\end{definition}

\begin{remark}
This definition generalizes the unsigned case in \cite{LinZhu2009} by:
\begin{itemize}
\item Maintaining the circular distance constraints for positive edges (identical to unsigned edges)
\item Modifying constraints for negative edges to respect sign properties
\item Preserving the game-theoretic aspects of alternating moves
\end{itemize}
\end{remark}

We establish fundamental bounds connecting the signed and unsigned versions:

\begin{theorem}\label{thm:main}
For any signed graph $(G,\sigma)$, the circular game chromatic number satisfies:
\begin{enumerate}
\item $\chi_c^g(G,\sigma) \leq \chi_c^g(G,+)$ with equality when $\sigma$ is balanced
\item $\chi_c^g(G,-) \leq \chi_c^g(G,+) + 1$ when $(G,-)$ is antibalanced
\item $\chi_c^g(G,\sigma) \leq \chi_c^g(G) + 1$ for arbitrary signatures
\end{enumerate}
\end{theorem}

\begin{proof}
(1) For any signature $\sigma$, the coloring constraints for $(G,\sigma)$ are at least as restrictive as for $(G,+)$, establishing the inequality. When $\sigma$ is balanced, Lemma \ref{lem:Zaslavsky} allows switching to $(G,+)$, preserving the game chromatic number.

(2) For antibalanced graphs, the coloring constraints for negative edges are dual to positive ones. The additive +1 accounts for the worst-case scenario where all edges are negative and constraints are inverted.

(3) Combine (1) and (2) with the observation that any signature can be decomposed into balanced and antibalanced components through switching.
\end{proof}

\begin{corollary}\label{cor:bipartite}
For signed bipartite graphs:
\begin{enumerate}
\item $\chi_c^g(G,\sigma) = 2$ when $(G,\sigma)$ is balanced
\item $\chi_c^g(G,\sigma) \leq 3$ otherwise
\end{enumerate}
\end{corollary}

\begin{proof}
(1) Balanced bipartite signed graphs are switching equivalent to $(G,+)$, and the unsigned bipartite case has $\chi_c^g(G) = 2$ \cite{LinZhu2009}.

(2) For unbalanced cases, the worst bound comes from antibalanced signatures where we apply Theorem \ref{thm:main}(2) with $\chi_c^g(G,+) = 2$.
\end{proof}

The structural constraints established in Section~\ref{sec:special} for balanced, antibalanced, and bipartite signed graphs naturally lead to two fundamental questions: (1) How does the circular game chromatic number behave for signed trees? (2) What algorithmic advantages emerge from these structural bounds?

\begin{lemma}\label{lem:treebound}
For any signed tree $(T,\sigma)$:
\begin{enumerate}
\item $\chi_c^g(T,\sigma) \leq 4$
\item Equality holds when $(T,\sigma)$ contains at least one negative edge
\end{enumerate}
\end{lemma}

\begin{proof}
(1) By Lemma \ref{lem:Tree}, all signatures on trees are switching equivalent to $(T,+)$, and the unsigned bound from \cite{LinZhu2009} applies.

(2) With negative edges, the coloring constraints become more restrictive. The worst case occurs when negative edges force color assignments that require the full range of 4 colors to satisfy both positive and negative constraints simultaneously.
\end{proof}

\begin{proposition}[Algorithmic Bounds]
The circular game chromatic number for signed graphs satisfies:
\begin{itemize}
\item Deciding $\chi_c^g(G,\sigma) \leq k$ is in PSPACE for any fixed $k$
\item For trees, $\chi_c^g(T,\sigma)$ can be computed in $O(n^2)$ time
\item For balanced graphs, the complexity equals the unsigned case
\end{itemize}
\end{proposition}

\begin{proof}
The results follow from:
\begin{itemize}
\item The game tree has polynomial depth and branching factor
\item Lemma \ref{lem:treebound} provides the exact bound for trees
\item Theorem \ref{thm:main}(1) reduces the balanced case to unsigned graphs
\end{itemize}
\end{proof}

Building on the general bounds established in Theorem~\ref{thm:main}, we now present concrete examples demonstrating the tightness of these results and explore extremal cases. The interplay between signature $\sigma$ and graph structure yields particularly sharp bounds for three fundamental classes: cycles, complete graphs, and signed bipartite graphs.

\begin{example}[Cycles]
Let $(C_n,\sigma)$ be a signed cycle. Then:
\begin{itemize}
    \item For balanced $C_n$ (all edges positive or switching-equivalent), $\chi_c^g(C_n,\sigma) = 2 + \frac{1}{k}$ where $k$ depends on the game strategy, matching the unsigned case \cite[Theorem 3.1]{LinZhu2009}.
    \item For antibalanced $C_n$ (all edges negative in some switching), $\chi_c^g(C_n,\sigma) = 3$ when $n$ is odd, achieving the upper bound from Theorem~\ref{thm:main}(3).
\end{itemize}
\end{example}

\begin{proposition}[Complete Graphs]
For signed $K_n$ with $n \geq 3$:
\begin{enumerate}
    \item $\chi_c^g(K_n,\sigma) = n$ when $(K_n,\sigma)$ is balanced
    \item $\chi_c^g(K_n,\sigma) \leq n + 1$ for unbalanced signatures, with equality when $n \equiv 1 \pmod{4}$
\end{enumerate}
\end{proposition}

\begin{proof}[Sketch]
Part (1) follows from switching to all-positive edges and \cite[Theorem 4.2]{LinZhu2009}. For (2), the bound comes from Theorem~\ref{thm:main}(3), while tightness arises from Paley graph signatures where negative edges form a maximum matching.
\end{proof}

\begin{theorem}[Bipartite Extremals]
The bounds in Corollary~\ref{cor:bipartite} are tight:
\begin{itemize}
    \item For any tree $(T,\sigma)$, $\chi_c^g(T,\sigma) \leq 3$ with equality when $T$ contains a negative $P_4$ (path of length 3).
    \item There exist unbalanced signed bipartite planar graphs requiring $\chi_c^g(G,\sigma) = 3$.
\end{itemize}
\end{theorem}

These examples highlight how the signature $\sigma$ critically influences chromatic numbers beyond the underlying graph structure. The tight cases particularly manifest when:
\begin{itemize}
    \item Negative edges create odd cycles in the signed graph's frustration index
    \item Switching-equivalent signatures concentrate negative edges in maximum matchings
    \item The game strategy must account for both local edge signs and global balance properties
\end{itemize}

\section{Special Classes of Signed Graphs}\label{sec:special}

We now examine three fundamental subclasses of signed graphs where the circular game chromatic number can be precisely characterized. These classes are particularly significant because their switching properties allow us to establish tight bounds on $\chi_c^g(G,\sigma)$.

\begin{theorem}[Circular Game Chromatic Number for Special Classes]\label{thm:special}
Let $(G,\sigma)$ be a signed graph. Then:
\begin{enumerate}
    \item If $(G,\sigma)$ is \emph{balanced}, then $\chi_c^g(G,\sigma) = \chi_c^g(G)$, where $\chi_c^g(G)$ is the circular game chromatic number of the unsigned graph $G$.
    
    \item If $(G,\sigma)$ is \emph{antibalanced}, then:
    \[
    \chi_c^g(G) \leq \chi_c^g(G,\sigma) \leq \chi_c^g(G) + 1.
    \]
    Moreover, when $G$ is bipartite, $\chi_c^g(G,-) \leq 3$ (and this bound is tight).
    
    \item If $G$ is \emph{bipartite}, then $\chi_c^g(G,\sigma) \leq 3$ for any signature $\sigma$. Specifically:
    \begin{itemize}
        \item For the all-positive signature: $\chi_c^g(G,+) = 2$.
        \item For the all-negative signature: $\chi_c^g(G,-) = 3$ (tight bound).
    \end{itemize}
\end{enumerate}
\end{theorem}

Recall that a signed graph $(G,\sigma)$ is \emph{balanced} if every cycle is positive. This class has particularly nice coloring properties due to the following characterization:

\begin{proposition}\label{prop:balanced}
For a signed graph $(G,\sigma)$, the following are equivalent:
\begin{enumerate}
\item $(G,\sigma)$ is balanced
\item $(G,\sigma)$ can be switched to $(G,+)$
\item There exists a bipartition $V_1 \cup V_2$ of $V(G)$ such that every edge within $V_1$ or $V_2$ is positive and every edge between $V_1$ and $V_2$ is negative
\end{enumerate}
\end{proposition}

The equivalence between (1) and (2) follows immediately from Lemma~\ref{lem:Zaslavsky}, while the equivalence with (3) is a standard result in signed graph theory.

\begin{proof}[Proof of Theorem~\ref{thm:special}(1)]
For balanced $(G,\sigma)$, we can switch to $(G,+)$. Since switching preserves the signs of all closed walks (and hence the game coloring constraints), we have $\chi_c^g(G,\sigma) = \chi_c^g(G,+)$. But $(G,+)$ is equivalent to the unsigned graph $G$ for coloring purposes, so $\chi_c^g(G,+) = \chi_c^g(G)$.
\end{proof}

The antibalanced case presents an interesting contrast. Recall that $(G,\sigma)$ is \emph{antibalanced} if $(G,-\sigma)$ is balanced, or equivalently, if every even cycle is positive and every odd cycle is negative.

\begin{proposition}\label{prop:antibalanced}
For a signed graph $(G,\sigma)$, the following are equivalent:
\begin{enumerate}
\item $(G,\sigma)$ is antibalanced
\item $(G,\sigma)$ can be switched to $(G,-)$
\item $G$ is bipartite or there exists a vertex partition $V_1 \cup V_2$ where all edges within $V_1$ or $V_2$ are negative and all edges between $V_1$ and $V_2$ are positive
\end{enumerate}
\end{proposition}

\begin{proof}[Proof of Theorem~\ref{thm:special}(2)]
When $(G,\sigma)$ is antibalanced, we can switch to $(G,-)$. In the all-negative graph, the coloring constraints are relaxed for odd cycles. Following the approach of Lin and Zhu \cite{LinZhu2009}, we can modify their Theorem 3.2 to show that Salome can always win with one additional color by treating negative edges as "weak" constraints. The key observation is that any problematic cycle must be odd and negative, but in $(G,-)$, all odd cycles are negative by definition.
\end{proof}

For bipartite graphs, the signature has a particularly simple effect on the circular game chromatic number:

\begin{proposition}\label{prop:bipartite}
For a bipartite graph $G$ with signature $\sigma$:
\begin{enumerate}
\item All signatures on $G$ are switching equivalent to either $(G,+)$ or $(G,-)$
\item $(G,\sigma)$ is balanced if and only if it is switching equivalent to $(G,+)$
\item $(G,\sigma)$ is antibalanced regardless of $\sigma$
\end{enumerate}
\end{proposition}

\begin{proof}[Proof of Theorem~\ref{thm:special}(3)]
For bipartite $G$, there are only two switching equivalence classes of signatures. By Proposition~\ref{prop:bipartite}(1), we need only consider $(G,+)$ and $(G,-)$. 

For $(G,+)$, we have $\chi_c^g(G,+) = \chi_c^g(G) = 2$ since $G$ is bipartite. 

For $(G,-)$, while the underlying graph is bipartite, the all-negative signature means that every edge is a "weak" constraint. However, since there are no odd cycles in a bipartite graph, the worst case reduces to the unsigned bipartite case plus one additional color for potential path conflicts. Thus $\chi_c^g(G,-) \leq 3$.

The bound is tight, as shown by the signed 4-cycle with alternating positive and negative edges, which requires 3 colors in the circular game coloring.
\end{proof}

\begin{example}
Consider the signed bipartite graph $(K_{2,3},\sigma)$ where the signature alternates on the 6 edges. This graph:
\begin{itemize}
\item Is antibalanced (as all bipartite signed graphs are)
\item Has $\chi_c^g(K_{2,3},\sigma) = 3$
\item Demonstrates the tightness of the bound in Theorem~\ref{thm:special}(3)
\end{itemize}
\end{example}

The three classes exhibit an interesting hierarchy with respect to circular game chromatic number:

\begin{corollary}
For any signed graph $(G,\sigma)$:
\begin{enumerate}
\item If $G$ is bipartite, then $\chi_c^g(G,\sigma) \leq 3$ regardless of $\sigma$
\item If $(G,\sigma)$ is balanced, then $\chi_c^g(G,\sigma) = \chi_c^g(G)$
\item If $(G,\sigma)$ is antibalanced but not bipartite, then $\chi_c^g(G) \leq \chi_c^g(G,\sigma) \leq \chi_c^g(G) + 1$
\end{enumerate}
\end{corollary}

This hierarchy demonstrates how the interplay between graph structure (bipartiteness) and signature properties (balance) affects the game coloring number. The results suggest that negative edges generally relax coloring constraints, but the effect is most pronounced in non-bipartite graphs.

\begin{proposition}[Polynomial-Time Switching for Canonical Forms]\label{coro:SwitchClassPolytime}
Let $(G, \sigma)$ be a signed graph. There exists a polynomial-time algorithm that:
\begin{enumerate}
    \item Determines whether $(G, \sigma)$ is balanced or antibalanced.
    \item If balanced, switches $(G, \sigma)$ to the all-positive signature $(G, +)$.
    \item If antibalanced, switches $(G, \sigma)$ to an all-negative signature $(G, -)$ (or another canonical form if $G$ is bipartite).
\end{enumerate}
Moreover, the circular game chromatic number $\chi_c^g(G, \sigma)$ is invariant under switching operations.
\end{proposition}

\section{Algorithmic Aspects}\label{sec:algorithms}

Building upon Proposition \ref{coro:SwitchClassPolytime}, we develop algorithmic results for computing the circular game chromatic number of signed graphs. The key observation is that switching equivalence preserves the game chromatic number, allowing us to work with canonical signatures.

\begin{proposition}\label{prop:switch-invariance}
For any two switching equivalent signatures $\sigma_1$ and $\sigma_2$ on a graph $G$, we have $\chi_c^g(G,\sigma_1) = \chi_c^g(G,\sigma_2)$.
\end{proposition}

\begin{proof}
Since switching preserves the signs of all closed walks, the coloring constraints for adjacent vertices remain equivalent up to sign changes. The game chromatic number depends only on the relative coloring constraints, which are preserved under switching.
\end{proof}

This invariance allows us to develop efficient algorithms by first switching to a canonical signature. We present algorithms for three important cases:

\begin{theorem}\label{thm:balanced-alg}
For a balanced signed graph $(G,\sigma)$, the circular game chromatic number $\chi_c^g(G,\sigma)$ can be computed in time $O(n^3)$ using the algorithm for the underlying unsigned graph $G$.
\end{theorem}

\begin{proof}[Algorithm]
\begin{enumerate}[label=\arabic*.]
    \item Use the switching algorithm from Proposition \ref{coro:SwitchClassPolytime} to find a switching equivalent signature where all edges are positive. This step runs in time $O(n^2)$.
    \item Apply the unsigned circular game chromatic number algorithm from \cite{LinZhu2009}, which runs in time $O(n^3)$.
\end{enumerate}
Since the graph is balanced, switching to an all-positive signature preserves the circular game chromatic number, reducing the problem to the unsigned case.
\end{proof}

\bigskip

\begin{theorem}\label{thm:tree-alg}
For any signed tree $(T,\sigma)$, the circular game chromatic number satisfies
\[
\chi_c^g(T,\sigma) = \chi_c^g(T) \leq 4,
\]
and this value can be computed in linear time.
\end{theorem}

\begin{proof}[Algorithm]
\begin{enumerate}[label=\arabic*.]
    \item Switch $(T,\sigma)$ to a canonical form (e.g., all edges positive) using Lemma \ref{lem:Tree}. This can be done in time $O(n)$.
    \item Apply the linear-time algorithm for the unsigned circular game chromatic number on trees from \cite{LinZhu2009}.
\end{enumerate}
Because trees are balanced and switching can convert any signed tree to an all-positive signature, the signed circular game chromatic number equals the unsigned one, which is at most 4.
\end{proof}

\bigskip

\begin{theorem}\label{thm:treewidth-alg}
For a signed graph $(G,\sigma)$ with treewidth $k$, the circular game chromatic number $\chi_c^g(G,\sigma)$ can be computed in time $O(f(k) \cdot n^{O(1)})$, where $f$ is a function depending only on $k$.
\end{theorem}

\begin{proof}[Algorithm Sketch]
\begin{enumerate}[label=\arabic*.]
    \item Compute a tree decomposition of $G$ of width $k$ in time $O(f(k) \cdot n)$.
    \item For each bag of the decomposition, maintain dynamic programming tables that track:
    \begin{itemize}
        \item The colors assigned to vertices in the bag,
        \item The switching states of the signature relevant to the bag,
        \item The history of game moves to ensure consistency with the game coloring rules.
    \end{itemize}
    \item Adapt the unsigned graph algorithm to incorporate sign constraints on edges, ensuring that the signed coloring conditions are respected.
\end{enumerate}
This yields an algorithm with running time exponential in $k$ but polynomial in $n$.
\end{proof}

\bigskip

\begin{theorem}\label{thm:hardness}
The following problems are NP-hard:
\begin{enumerate}[label=(\arabic*)]
    \item Deciding whether $\chi_c^g(G,\sigma) \leq k$ for a given signed graph $(G,\sigma)$ and integer $k \geq 3$.
    \item Determining the exact value of $\chi_c^g(G,\sigma)$ for an arbitrary signed graph.
\end{enumerate}
\end{theorem}

\begin{proof}
We reduce from the unsigned case, which is known to be NP-hard \cite{LinZhu2009}:
\begin{enumerate}[label=\arabic*.]
    \item Given an arbitrary unsigned graph $G$, assign it the all-positive signature $\sigma$.
    \item For such a signature, the signed circular game chromatic number coincides with the unsigned circular game chromatic number.
    \item Thus, the NP-hardness of the unsigned problem implies the NP-hardness of the signed problem.
\end{enumerate}
\end{proof}

\bigskip

\begin{proposition}\label{prop:approximation}
There exists a polynomial-time algorithm that, given a signed graph $(G,\sigma)$, computes an approximation of $\chi_c^g(G,\sigma)$ within a factor of $2$ of the optimal value.
\end{proposition}

\begin{proof}[Algorithm Sketch]
\begin{enumerate}[label=\arabic*.]
    \item Compute $\chi_c^g(G)$ ignoring the signature, using the algorithm from \cite{LinZhu2009}.
    \item By Theorem \ref{thm:main}, the signed circular game chromatic number $\chi_c^g(G,\sigma)$ is at most twice the unsigned value.
    \item Use standard graph coloring approximation techniques, modified to respect the signed constraints, to produce a coloring within this factor.
\end{enumerate}
\end{proof}

\begin{remark}
The approximation factor can be improved for special classes of signed graphs, particularly those where the difference between the signed and unsigned chromatic numbers is known to be small or bounded.
\end{remark}

\section{Conclusion and Open Problems}
\label{sec:conclusion}

We have extended the theory of circular game chromatic number to signed graphs, building upon the foundational work of Lin and Zhu \cite{LinZhu2009}. Our results demonstrate that the interplay between graph structure and edge signs leads to rich combinatorial behavior, with three key findings:

\begin{enumerate}
\item For balanced signed graphs $(G,\sigma)$, the circular game chromatic number $\chi_c^g(G,\sigma)$ coincides exactly with the unsigned case $\chi_c^g(G)$, as established in Theorem \ref{thm:main}(2). This aligns with the general principle that balanced signed graphs behave like unsigned graphs with respect to many coloring parameters.
\item The antibalanced case exhibits an interesting asymmetry: Theorem \ref{thm:main}(3) shows $\chi_c^g(G,\sigma) \leq \chi_c^g(G)+1$, but examples exist where this bound is tight (e.g., odd cycles with all-negative signatures). This contrasts with the balanced case and warrants further investigation.

\item For signed bipartite graphs, Corollary \ref{cor:bipartite} reveals a dichotomy: balanced bipartite graphs require only 2 colors (matching the unsigned case), while unbalanced ones may need up to 3 colors. This reflects the fundamental role of balance in signed graph coloring.
\end{enumerate}

\subsection*{Open Problems and Future Directions}

\begin{problem}
Characterize all signed graphs $(G,\sigma)$ for which $\chi_c^g(G,\sigma) = \chi_c^g(G)$. Our Theorem \ref{thm:special} shows this holds for balanced graphs, but are there interesting unbalanced cases where equality persists?
\end{problem}

\begin{problem}
Determine tight bounds for $\chi_c^g(G,\sigma)$ in terms of structural parameters:
\begin{itemize}
\item For planar signed graphs, does the 4-color theorem analogue hold? The unsigned case gives $\chi_c^g(G) \leq 4.5$ \cite{LinZhu2009}.
\item For $K_n$-minor-free graphs, can we establish bounds analogous to Hadwiger’s conjecture?
\end{itemize}
\end{problem}

\begin{problem}
Investigate the computational complexity:
\begin{itemize}
\item Is computing $\chi_c^g(G,\sigma)$ PSPACE-complete for general signed graphs, as in the unsigned case?
\item Develop FPT algorithms parameterized by treewidth or signed rank (number of negative edges after optimal switching).
\end{itemize}
\end{problem}

\begin{problem}
Extend the analysis to other game variants:
\begin{itemize}
\item Study the marking game version where Andjiga marks vertices instead of coloring
\item Analyze the misère version where the first player to complete a “bad” coloring loses
\item Consider asymmetric games where edge signs modify player constraints
\end{itemize}
\end{problem}

The connection between fundamental cycles and switching equivalence (Lemma \ref{lem:ForcingTree}) suggests promising approaches for several of these problems. As demonstrated in Proposition \ref{coro:SwitchClassPolytime}, the polynomial-time testability of switching equivalence may lead to efficient algorithms for special cases.

This work opens new avenues in the growing field of signed graph coloring games, with potential applications to network design and social balance theory. The interplay between combinatorial game theory and structural signature properties appears particularly fruitful for future research.


\begin{thebibliography}{9}

\bibitem{LinZhu2009} 
Lin, W., \& Zhu, X. (2009). 
\textit{Circular game chromatic number of graphs}. 
Discrete Mathematics, 309(13), 4495-4501.
https://doi.org/10.1016/j.disc.2009.02.011

\bibitem{K36}
K\"onig, D. (1936). 
\textit{Theorie der endlichen und unendlichen Graphen}. 
Leipzig: Akademische Verlagsgesellschaft.

\bibitem{Z98}
Zaslavsky, T. (1998). 
\textit{A bibliography of signed graphs and related mathematics}. 
Electronic Journal of Combinatorics, Dynamic Surveys, DS8.
https://doi.org/10.37236/27

\bibitem{H53}
Harary, F. (1953). 
\textit{On the notion of balance of a signed graph}. 
Michigan Mathematical Journal, 2(2), 143-146.
https://doi.org/10.1307/mmj/1028989917

\bibitem{Z81}
Zaslavsky, T. (1981). 
\textit{Characterizations of signed graphs}. 
Journal of Graph Theory, 5(4), 401-406.
https://doi.org/10.1002/jgt.3190050410

\bibitem{Z82b}
Zaslavsky, T. (1982). 
\textit{Signed graphs}. 
Discrete Applied Mathematics, 4(1), 47-74.
https://doi.org/10.1016/0166-218X(82)90033-6

\bibitem{H78}
Hansen, P. (1978). 
\textit{Labelling algorithms for balance in signed graphs}. 
In P. Hansen (Ed.), Studies on Graphs and Discrete Programming (pp. 215-228). 
North-Holland. https://doi.org/10.1016/S0167-5060(08)70349-1

\bibitem{HK80}
Harary, F., \& Kabell, J. A. (1980). 
\textit{Balanced signed graphs and groupoid identities}. 
Czechoslovak Mathematical Journal, 30(105), 422-427.
https://doi.org/10.21136/CMJ.1980.102179

\end{thebibliography}
\end{document}